\documentclass{amsart}
\newtheorem{theorem}{Theorem}[section]
\newtheorem{lemma}[theorem]{Lemma}
\theoremstyle{definition}
\newtheorem{definition}[theorem]{Definition}

\newtheorem{corollary}[theorem]{Corollary}
\newtheorem{proposition}[theorem]{Proposition}
\theoremstyle{remark}

\numberwithin{equation}{section}

\begin{document}

\def\C{\mathbb C}
\def\R{\mathbb R}
\def\X{\mathbb X}
\def\Z{\mathbb Z}
\def\Y{\mathbb Y}
\def\Z{\mathbb Z}
\def\N{\mathbb N}
\def\cal{\mathcal}

\def\cal{\mathcal}
\def\b{\mathcal B}
\def\c{\mathcal C}
\def\cc{\mathbb C}
\def\x{\mathbb X}
\def\r{\mathbb R}
\def\uu{(U(t,s))_{t\ge s}}
\def\vv{(V(t,s))_{t\ge s}}
\def\xx{(X(t,s))_{t\ge s}}
\def\yy{(Y(t,s))_{t\ge s}}
\def\zz{(Z(t,s))_{t\ge s}}
\def\ss{(S(t))_{t\ge 0}}
\def\tt{(T(t,s))_{t\ge s}}
\def\rr{(R(t))_{t\ge 0}}

\title[Asymptotic Behavior of Solutions of FDE]{
Asymptotic Behavior of Solutions of periodic linear partial functional differential equations on the half line}

\author{Vu Trong Luong}
\address{Department of Mathematics,  Tay Bac University, Son La City, Son La, Vietnam}
\email{vutrongluong@gmail.com}
\author{Nguyen Huu Tri}
\address{Division of Mathematics, Truong THPT Trung Van, Phuong Trung Van, Nam Tu Liem, Hanoi, Vietnam}
\email{huutri1007@gmail.com}
\author{Nguyen Van Minh}

\address{Department of Mathematics and Statistics, University of Arkansas at Little Rock, 2801 S University Ave, Little Rock, AR 72204. USA}
\email{mvnguyen1@ualr.edu}
\thanks{This paper is partially supported by the grant B2018-TTB-11 of the Ministry of Education and Training of Vietnam}

\date{\today}

\begin{abstract}
We study conditions for the abstract linear functional differential 
equation $\dot{x}=Ax+F(t)x_t+f(t), t\ge 0$ to have asymptotic almost periodic solutions, where
$F(\cdot )$ is periodic, $f$ is asymptotic almost periodic. The main conditions are stated in terms of 
the spectrum of the monodromy operator associated with the equation and the circular spectrum of the forcing term $f$. 
The obtained results extend recent results on the subject. 
\end{abstract}
\maketitle

\section{Introduction} \label{section 1}
In this paper we consider the asymptotic behavior of solutions to equations of the form
\begin{equation}\label{FDE}
\frac{dx(t)}{dt}=Ax(t)+F(t)x_t +f(t),\ \ x \in {\X}, t \ge 0,
\end{equation}
where the (unbounded) linear operator $A$ generates a strongly continuous semigroup, $x_t\in C_r:=C([-r,0],{\X})$, $x_t(\theta ):=x(t+\theta )$, 
$r>0$ is a given positive real number,
$F(t)\varphi :=\int^0_{-r}d\eta (t,s)\varphi (s), \ \forall \varphi \in C_r$,
$\eta (t,\cdot ) : C_r\to L({\X})$ is periodic and continuous in $t$, of bounded variation, and $f$ is a ${\X}$-valued asymptotic almost periodic function on the half line. Existence of asymptotic almost periodic solutions to equations on the half line is an interesting topic in the qualitative theory of differential equations. The reader is referred to \cite{bathutrab,che,fin} and the references therein for more information. If one uses the spectral theory of functions to study this problem, then a big challenge arises when one estimates the spectrum of mild solutions due to the non-autonomousness of the equations. In a recent paper \cite{mingasste}, a concept of circular spectrum of functions on the whole line was introduced, and demonstrated as alternative approach to the problem. In this paper we further this idea for functions defined on the half line to study the asymptotic almost periodicity of solutions of (\ref{FDE}). The main results obtained in this paper are Theorems \ref{the 1}, \ref{the 2} where sufficient conditions for the asymptotic almost periodicity of solutions are given in terms of spectral properties of the input $f$ and the monodromy operator associated with the periodic homogeneous equations $x'=A(t)x(t)+F(t)x_t$. We note that in the paper \cite{luomin} the circular spectrum of functions on the real line is employed to study the asymptotic behavior of solutions of equations defined on the whole line. This being said, our study in this paper will complement that of \cite{luomin}. The last but not the least is that our assumption of period 1 on the evolution processes considered in the paper does not constitute any restriction of the obtained results.

\section{Preliminaries} \label{section 2}
\subsection{Notation}
Throughout the paper we will use the following notations: 
${\N}, {\Z},
{\R}, {\C}$ stand for 
the sets of natural, integer, real, complex numbers, respectively. 
$\Gamma$ denotes the unit circle
in the complex plane ${\C}$. 
${\X}$ will denote
a given complex Banach space. Given two Banach spaces 
${\X},{\Y}$ by $L({\X},{\Y})$  we will 
denote the space of all bounded linear operators from  ${\X}$ 
to  ${\Y}$. As usual,  $\sigma (T), \rho (T), R(\lambda ,T)$ 
are the notations of the spectrum, resolvent set and resolvent of the operator
 $T$.
The notations \ $BC({\R},{\X}), BUC({\R},{\X}),
AP({\X})$ \ will stand for the spaces of all ${\X}$-valued
bounded continuous, bounded uniformly continuous functions on ${\R}$ and its subspace of almost 
periodic (in Bohr's sense) functions, respectively.  Let us denote by 
\begin{eqnarray*}
C_0(\X) &:=& \{ f\in BUC(\R^+,\X): \ \lim_{t\to +\infty }f(t)=0 \} \\
AP(\X) &:=& \{ f\in BUC(\R,\X): \ f \ \mbox{is almost periodic}\} \\
AP(\R^+,\X) &:=& \{ g: \ g=f|_{\R^+} \ \mbox{for some} \ f \in AP(\X)\} \\
AAP(\R^+,\X) &:=& AP(\R^+,\X) \oplus C_0(\X) .
\end{eqnarray*}
Any element of $AAP(\R^+,\X)$ is called an asymptotic almost periodic function with values in $\X$.
For a positive number $r$ we denote $C_r:= C([-r,0],{\X})$, If $x(\cdot )$ is a function define on $(a,b)$, then
$x_t(\theta ):=x(t+\theta )$, if $t+\theta \in (a,b)$. In this case $x_t\in C_r$. If the interval $(a,b)$ is large enough, so $x_t$ is defined on an interval, and  the notation $x_\cdot$ will be used later to denote the function $t\mapsto x_t$.

\subsection{Almost periodic functions}
A subset $E\subset {\R}$ is said to be {\it relatively dense} 
if there exists a number $l>0$ ({\it inclusion length})
such that every interval $[a,a+l]$ contains at least one point of $E$.
Let $f$ be a continuous function on ${\R}$ taking values in a complex Banach space ${\X}$.
$f$ is said to be {\it almost periodic in the sense of Bohr}
if to every $\epsilon >0$ there corresponds a
relatively dense set $T(\epsilon , f)$ ({\it of $\epsilon$-periods })
such that
$$
\sup_{t \in {\R}}\| f(t+\tau )-f(t)\| \le \epsilon  , \ \forall \tau \in T(\epsilon , f).
$$
If $f$ is almost periodic function, then (approximation theorem \cite[Chap. 2]{levzhi}) it can be approximated
uniformly on ${\R}$  by a sequence of trigonometric polynomials, i.e.,
a sequence of functions in $t\in {\R}$ of the form
\begin{equation*}\label{trig pol}
P_n(t):=\sum_{k=1}^{N(n)}a_{n,k}e^{i\lambda_{n,k}t}, \ \ n=1,2,...; \lambda_{n,k} \in {\R}, a_{n,k}\in {\X}, t\in {\R}.
\end{equation*}
Of course, every function which can be approximated by a sequence of trigonometric polynomials is almost periodic.
Specifically, the exponents of the trigonometric polynomials (i.e., the reals $\lambda_{n,k}$ in (\ref{trig pol})) can be chosen from the set of 
all reals $\lambda$ ({\it Fourier exponents}) such that the following integrals ({\it Fourier coefficients}) 
$$
a(\lambda , f):=\lim_{T\to \infty}\frac{1}{2T}\int^T_{-T}f(t)e^{-i\lambda t}dt
$$
are different from $0$. As is known, there are at most countably such reals $\lambda$, the set of which will
be denoted by $\sigma_b(f)$ and called {\it Bohr spectrum} of $f$.

\medskip
We define
\begin{equation*}
AP(\R^+,\X) := \{ f|_{\R^+}, f\in AP(\X)\} .
\end{equation*}
From the definition of almost periodicity it is easy to prove that
\begin{equation*}
\sup_{t\in \R} \| f(t) \| =\sup_{t\in \R^+} \| f(t)\| , \ \mbox{for all} \ f\in AP(\X) .
\end{equation*}
Therefore, the operator of restriction of a function from $AP(\X)$ to the half line $\R^+$ is actually an invertible isometry from $AP(\X)$ onto $AP(\R^+,\X)$. Later on we will sometimes identify the function $f\in AP(\X)$ with its restriction to the half line $\R^+$.

\subsection{Circular Spectra of Functions on the half line}
Recall that the translation semigroup $(S(t))_{t\ge 0}$ defined as
\begin{equation*}
[S(t)x](\xi ) := x(t+\xi ), \quad t\ge 0, \xi \ge 0, \ x\in BUC(\R^+,\X) 
\end{equation*}
leaves $BUC(\R^+,\X)$ invariant.

\begin{definition}
A subspace of ${\cal M}$ of $BUC (\R^+,\X)$ is said to be admissible if it of the form $AP^+_{\Lambda}(\X) \oplus C_0(\X)$, where $AP^+_{\Lambda }(\X)$ is the restrictions of all almost periodic functions to the positive half line with Bohr spectrum contained in a closed set $\Lambda$ of real numbers.
\end{definition}
Note that for any given set of reals $\Lambda$ the subspace $AP^+_{\Lambda}(\X) $ of $BUC(\R^+,\X)$ is closed, so is $AP^+_{\Lambda}(\X) \oplus C_0(\X)$. This means that any admissible subspace ${\cal M}$ is a closed subspace of $BUC(\R^+,\X)$.

For a given admissible subspace ${\cal M}$ let us consider the quotient space $\Y :=BUC (\R^+,\X)/{\cal M} $. Since the translation semigroup in $BUC (\R^+,\X)$ leaves the closed subspace ${\cal M}$ invariant, this induces a strongly continuous semigroup in $\Y$, denoted by $(\bar S(t))_{t\ge 0}$. The following lemma is actually known as a remark in \cite{bathutrab}. However, for the completeness we will give a proof below.
\begin{lemma}\label{lem biinvariant}
Let ${\cal M}$ be an admissible subspace of $BUC (\R^+,\X)$. Under the above notations, the semigroup $(\bar S(t))_{t\ge 0}$ can be extended to a group of isometries in $\Y$.
\end{lemma}
\begin{proof} 
First, for each $t>0$ we will prove that $\bar S(t)$ is invertible. To this end, given an element $[g]\in \Y$, we will show that there exists a unique class $[f] \subset BUC(R^+,\X)$ as an element of $\Y$ such that $\bar S(t)[f]=[g]$. 
Note that the function $f$ defined as
\begin{equation*}
f(s):= \begin{cases} g(s-t), \ \mbox{if} \ s\ge t \\
g(0), \ \mbox{if} \  0\le s <t 
\end{cases}
\end{equation*}
is in $BUC(R^+,\X)$, and satisfies $S(t)f=g$, so $\bar S(t)[f]=[g]$. 
For the uniqueness we suppose that $\bar S(t) [k]=[0]$ for some class $[k]$. Then, $[k]=[0]$. In fact, if $k(\cdot )$ is a function in $BUC(\R^+,\X)$ such that $k(\cdot + t)$ is in  ${\cal M}:=AP^+_{\Lambda}(\X) \oplus C_0(\X)$, that is $k(\cdot +t) = k_1(\cdot ) + k_2(\cdot )$ where $k_1\in AP^+_{\Lambda }(\X), k_2 \in C_0(\X)$. 
It is well known that $k_2(\cdot )$ is the restriction of an almost periodic function on the real line to the positive half line, so $k_2$ may denote for an almost periodic function on $\R$. We will extend $k_2$ to the interval $[-t, \infty)$ by setting
\begin{equation*}
k_2(s):= \begin{cases} k_2(s), \ \mbox{if} \ s\ge 0 \\
k(s+t)-k_1(s), \ \mbox{if} \  -t\le s <0 .
\end{cases}
\end{equation*}
Obviously, $k_1(\cdot -t) \in AP^+_{\Lambda} (\X)$, and $k_2(\cdot -t) \in C_0(\X)$. Therefore, as $k(s) = k_1(s-t) +k_2(s-t)$, $k\in {\cal M}$, that í $[k]=[0]$, the uniqueness is proved. Therefore, $(\bar S(t))_{t\ge 0}$ is extendable to a strongly continuous group.

\medskip
Next, we prove that $(\bar S(t))_{t\in \R}$ consists of isometries. From the above argument for $t\ge 0$ we have
\begin{eqnarray*}
\| \bar S(t)[f]\| &=& \inf_{ g\in [f]} \| S(t)g\| \\
&=& \inf_{g\in [f]}  \sup_{s\ge 0} \| g(s+t)\| \\
&\le& \inf_{g\in [f]}  \sup_{s\ge 0} \| g(s)\| \\
&=& \inf_{g\in [f]} \| g\| \\
&=& \| [f]\| .
\end{eqnarray*}
Therefore, if $t\ge 0$, $\| \bar S(t)\| \le 1$. Next, for each function $f\in BUC(R^+,\X)$ and $g\in {\cal M}$ we define
$$
h (s):= \begin{cases} g(s-t), \ \mbox{if} \ s\ge t,\\
g(0)+f(s)-f(t), \ \mbox{if} \ 0\le s < t .
\end{cases}
$$
Then, $h\in {\cal M}$, and 
\begin{eqnarray*}
\| f-h\| &=& \sup_{s\ge 0} \| f(s)-h(s) \| \\
&=& \sup_{s\ge t} \| f(s)-g(s-t)\| =\sup_{s\ge 0} \| f(s+t)-g(s)\| \\
&=& \| S(t)f-g\| .
\end{eqnarray*}
Therefore,
\begin{eqnarray*}
\| [f]\| &= & \inf_{k\in {\cal M}} \| f-k\| \le \| f-h\| \\
&\le& \| S(t)f-g\| .
\end{eqnarray*}
Since $g$ is arbitrary this shows that
\begin{equation*}
\| [f]\| \le \inf_{g\in {\cal M}} \| S(t)f-g\| =\| [S(t)f]\| ,
\end{equation*}
and hence, $\| [f] \| \le \| \bar S(t)[f]\| $ for each $t>0$, $f\in BUC(R^+,\X)$. In view of the above inequality, $\| \bar S(t)[f]\| =\| [f]\| $ for all $t\ge 0$. If $t<0$, then, for each $f\in BUC(R^+,\X)$
\begin{eqnarray*}
\| [f]\| &=& \| \bar S(-t)\bar S(t) [f]\| =\| \bar S(t)[f]\| .
\end{eqnarray*}
Therefore, the group $(\bar S(t))_{t\in \R}$ is a strongly continuous group of isometries.
\end{proof}

\medskip
Let ${\cal M}$ be a fixed admissible subspace of $BUC(R^+,\X)$. For each $x(\cdot )\in BUC(\R,^+\X)$ let us consider the complex function ${\cal S}x(\lambda )$ in $\lambda \in \C \backslash \Gamma$ defined as
\begin{equation*}\label{transform}
{\cal S}x(\lambda ):= R(\lambda ,\bar S)x,\quad \lambda \in \C \backslash
\Gamma ,
\end{equation*}
where $\bar S := \bar S(1)$.
Since $\bar S$ is an invertible isometry, this transform is an analytic function
in $\lambda \in \C \backslash \Gamma $. 

\begin{lemma}\label{lem 1-0}
We have the following estimate:
\begin{equation}\label{est1}
\| {\cal S}x(\lambda )\| \le \frac{\| \bar x\| }{|1-|\lambda||}, \ |\lambda | \not= 1.
\end{equation}
\end{lemma}
\begin{proof}
For any element $\bar y \in \Y$ , we have
\begin{eqnarray*}
\| (\lambda I-\bar S)\bar y\| &=& \| \lambda \bar y -\bar S\bar y\| \\
&\ge & | \| \lambda \bar y\| -\| \bar S\bar y\| | \\
&=& | |\lambda| \cdot \| \bar y\|  -\| \bar y \| | \\
&=& (1-| \lambda | )  \cdot \| \bar y\| .
\end{eqnarray*}
If $\lambda | \not = 1$,  
we choose $\bar y:= R(\lambda ,\bar S)\bar x$. Therefore,
\begin{eqnarray*}
\| \bar x\| &=& \| (\lambda I-\bar S)\bar R(\lambda ,\bar S)\bar x\| \\
&\ge & (1-| \lambda | )  \cdot \| R(\lambda ,\bar S)\bar x \| .
\end{eqnarray*}
This yields (\ref{est1}).
\end{proof}

\begin{definition}\rm
The {\it circular spectrum} of $x\in BUC(\R^+,\X)$ with respect to a given admissible subspace ${\cal M}$ is defined to
be the set of all $\xi_0\in \Gamma$ such that ${\cal S}x(\lambda )$
has no analytic extension into any neighborhood of $\xi_0$ in the
complex plane. This spectrum of $x$ is denoted by $\sigma_{\cal M} (x)$.
 When ${\cal M} =C_0(\X)$ it will be called for short {\it the spectrum of $g$} if this does not
cause any confusion. We will denote by $\rho (x)$ the set $\Gamma
\backslash \sigma (x)$.
\end{definition}

\medskip
The following lemma justifies the introduction of these concepts of spectra.
\begin{lemma}\label{lem 1}
Let ${\cal M}$ be an admissible subspace of $BUC(R^+,\X)$, and let $x(\cdot )\in BUC(\R^+,\X)$. Then, 
\begin{enumerate}
\item $x(\cdot )$ is in ${\cal M}$ if and only if $\sigma_{\cal M} (x) =\emptyset$,
\item If $Q$ is an operator in $BUC(\R^+,\X)$ that commutes with $S(1)$ and leaves ${\cal M}$ invariant, then for each $x(\cdot )\in BUC(\R^+,\X)$,
\begin{eqnarray}
\sigma _{\cal M}(Qx) &\subset & \sigma_{\cal M} (x).
\end{eqnarray}
\end{enumerate}
\end{lemma}
\begin{proof}
(i): If $x(\cdot )$ is in ${\cal M}$, then $\bar x =0$ in $\Y$. Therefore, $g(\lambda ):= R(\lambda , \bar S)\bar x=0$ for all $\lambda \in \C \backslash \Gamma$, so it has a natural extension to the whole complex plane, and thus $\sigma _{\cal M}(x)=\emptyset$. Conversely, let $\sigma_{\cal M} (x) =\emptyset$. Then, the function $g(\lambda ):= R(\lambda , \bar S)\bar x$ is an entire function. By (\ref{est1}), the function is bounded, so by Liouville Theorem it is a constant. Moreover, this constant must be zero also by (\ref{est1}). Since $R(\lambda ,\bar S)$ is an injective, this yields that $\bar x=0$. That is, $x(\cdot )$ is in ${\cal M}$.

\medskip
(ii): This claim is straightforward as $R(\lambda , \bar S )\bar Q g = \bar Q R(\lambda , \bar S)g$, where $\bar S$ and $\bar Q$ are operators induced by $S$ and $Q$ in the quotient spaces $BUC(\R^+,\X)/{\cal M}$.
\end{proof}

In the same lines as in \cite{min} the following version of a Gelfand's Theorem can be obtained:
\begin{lemma}\label{lem 2}
 Assume that $\bar x$ is any point in $\Y$, and the complex function ${\cal S}(\lambda
):= R(\lambda ,\bar S )\bar x$ has the point $\lambda = \xi_0\in \Gamma$ as an isolated singular point. Then, $\xi_0$ is either a removable singular point or a pole of first order.
\end{lemma}

Before proceeding we introduce a new notation:  let $0\not= z\in \C$ such that
$z= re^{i\varphi}$ with reals $r=|z|,\varphi$, and let $F(z)$ be any complex
function. Then, (with $s$ larger than $r$) we define
\begin{equation*}
\lim_{\lambda \downarrow z} F(\lambda ):= \lim_{s\downarrow r}F(s e^{i\varphi}).
\end{equation*}
As a consequence of Lemma \ref{lem 2} we have the following:
\begin{lemma}\label{lem 3}
Let $\xi_0\in\Gamma$ be an isolated singular point of ${\cal S}(\lambda )= R(\lambda ,\bar S)\bar x$ with a given $\bar x\in \Y$. Then, this singular point $\xi_0$ is removable provided that
\begin{equation}\label{erg}
\lim_{\lambda \downarrow \xi_0} (\lambda -\xi_0 )R(\lambda ,\bar S)\bar x =0 .
\end{equation}
\end{lemma}
As a consequence of the Lemma \ref{lem 3} we can prove the following result that will be used as a main tool later on:
\begin{theorem}\label{the tec}
Let ${\cal M}$ be an admissible subspace of $BUC(R^+,\X)$, and let $x(\cdot )$ be an element of $BUC(\R^+,\X)$ such that the set $\sigma_{\cal M}(x)$  is countable, and let the following condition hold for
each $\xi_0\in \sigma _{\cal M}( x)$  
\begin{equation}
\lim_{\lambda \downarrow \xi_0} (\lambda -\xi_0 )R(\lambda ,\bar
S)\bar x =0 .
\end{equation}
Then,  $x(\cdot )$ is in ${\cal M}$. 
\end{theorem}

\subsection{Beurling spectrum and almost periodic functions}
Below we will recall the concept of Beurling spectrum of a function. 
We denote by $F$ the Fourier transform, i.e.
\begin{equation}
({F}f)(s):= \int^{+\infty }_{-\infty }e^{-ist}f(t)dt
\end{equation}
$(s\in {\R}, f\in L^1({\R}))$. Then the {\it Beurling spectrum}\index{Beurling spectrum} 
of $u\in BUC({\R},{\X})$ is defined to be the following set
\begin{eqnarray*}
sp(u)&:=& \{ \xi \in {\R}: \forall \epsilon >0 \ \ \exists f\in L^1({\R}),\\
&& \hspace{1cm} supp{F}f \subset (\xi -\epsilon ,\xi +\epsilon ), f*u \not= 0\}
\end{eqnarray*}
where 
$$
f*u(s):=\int^{+\infty }_{-\infty }f(s-t)u(t)dt .
$$

\medskip
Throughout the paper, if $f\in AP(\X)$ we will use
the relation 
\begin{equation}\label{bohr}
sp(f)= \overline{\sigma_b(f)}.
\end{equation}
Let $g\in AP(\R^+,\X)$. Then, there exists a unique $g_\R \in AP(\X)$ whose restriction to $\R^+$ is exactly $g$. \begin{proposition}\label{pro}
Let $g\in AP(\R^+,\X)$. Then
\begin{equation}
\sigma (g)=  \overline{e^{i sp(g_\R )}}.
\end{equation}
\end{proposition}
\begin{proof}
By the Weak Spectral Mapping Theorem (see e.g. \cite{engnag}) we have
\begin{equation*}
\sigma (g_\R)= \overline{e^{i sp(g_\R )}},
\end{equation*}
where $\sigma (g_\R)$ denotes the circular spectrum of $g_\R$ as a function defined on the whole line $\R$ (see \cite{mingasste}).
Therefore, it suffices to show that
\begin{equation*}
\sigma (g) = \sigma (g_\R).
\end{equation*} 
For $|\lambda | \not =1$, by the isometry mentioned above (that is, the isometry $AP(\X) \ni f \mapsto f|_{\R^+}\in AP(\R^+,\X)$) we can identify $R(\lambda ,S)g$ with $R(\lambda ,S)g_\R$. That means, the complex function $R(\lambda ,S)g$ can be identified as $R(\lambda ,S)g_\R$ for any $\lambda \not = 1$. Suppose that $\lambda _0\in \Gamma$, and $R(\lambda ,S)g_\R$ has an analytic extension to a neighborhood of $\lambda_0$. We can take this extension for the complex function $R(\lambda ,S)g$, so the complex function $R(\lambda ,\bar S)\bar g$ also has an analytic extension to a neighborhood of $\lambda_0$.
Therefore, $\lambda_0\not \in \sigma (g)$.

Conversely, if $\lambda_0\in \Gamma$ and 
$ \lambda_0\not\in \sigma (g)$. By the definition, $R(\lambda , \bar S)\bar g$ has an analytic extension to a neighborhood of $\lambda _0$. Note that the projection map $p: BUC(\R^+,\X) \to \Y : =BUC(\R^+,\X)/C_0(\X)$ preserves norm if we restrict $p$ to $AP(\R^+,\X)$. Therefore we can identify $R(\lambda ,\bar S)\bar g$ with $R(\lambda ,S)g$ since all these functions are almost periodic. Hence, we can identify $R(\lambda , \bar S)\bar g$ with $R(\lambda , S)g_\R$, and thus, 
we can use the analytic extension of $R(\lambda ,\bar S)\bar g$ to be an analytic extension of $R(\lambda ,S)g_\R$.
 or $\lambda_0\not\in \sigma (g_\R$. The proposition is proved.
\end{proof}

\subsection{Periodic evolutionary processes}
\begin{definition}
Let $(U(t,s))_{t\ge s}$ be a two-parameter family of bounded operators in a Banach space $\X$. Then, it is called an evolutionary process if
\begin{enumerate}
\item $U(t,t)=I$\ \ for all \ $t\in {\R}$,
\item $U(t,s)U(s,r)=U(t,r)$\ \ for all \ \ $t\ge s\ge r$,
\item The map \ \ $(t,s)\mapsto U(t,s)x$\ \ is continuous for every fixed \ \ $x
\in {\X}$,
\item $\| U(t,s)\| < Ne^{\omega (t-s)} $ for 
some positive \ $N, \omega $ \ independent of $t \geq s$ .
\end{enumerate}
\end{definition}
An evolutionary process is called {\it 1-periodic} if
\begin{equation*}
U(t+1 ,s+1 )=U(t,s), \ \mbox{for all} \ t \ge s .
\end{equation*}

Recall that for a given 1-periodic evolutionary process
$(U(t,s))_{t\ge s}$ the following operator
\begin{equation*}
P(t):= U(t,t-1), t \in {\R}
\end{equation*}
is called {\it monodromy operator} (or sometime {\it period map,
Poincar\'e map}). Thus we have a family of monodromy operators. We
will denote $P:= P(0)$. The nonzero eigenvalues of $P(t)$  are
called {\it characteristic multipliers}. An important property of
monodromy operators is stated in the following lemma whose proof
can be found or modified from similar results in \cite{hen,hinnaiminshi}.
\begin{lemma}\label{lem 5.2}
Under the notation as above the following assertions hold:
\begin{enumerate}
\item $P(t+1) = P(t)$ for all $t$; characteristic multipliers are
independent of time, i.e. the nonzero eigenvalues of $P(t)$ coincide
with those of $P$, \item $\sigma (P(t)) \backslash \{0\}=\sigma (P)
\backslash \{0\}$, i.e., it is independent of $t$, \item If $\lambda
\in \rho (P)$, then the resolvent $R(\lambda , P(t))$ is strongly
continuous, \item If ${\cal P}$ denotes the operator of
multiplication by $P(t)$ in any one of the function spaces $BUC(\R^+,\X)$ or $AP(\R^+, \X)$, then
\begin{equation}\label{5.3}
\sigma ({\cal P}) \backslash \{0\} \subset \sigma (P)\backslash
\{0\}.
\end{equation}
\end{enumerate}
\end{lemma}

\section{Existence of almost periodic solutions of Eq.(\ref{FDE})}
\begin{definition}
A continuous function $u(\cdot )\in BUC(\R^+,\X)$ is said to be a mild solution on ${\R^+}$ of 
Eq.(\ref{FDE}) with initial $\phi \in C_r$
if $u_0 =\phi$ and for all $t > s \ge 0 $
\begin{equation}\label{mild solution}
u(t)=T(t-s)u(s)+\int^t_sT(t-\xi )[F(\xi )u_\xi +f(\xi )]d\xi 
\end{equation}
\end{definition}
It is well known that if $A$ generates a $C_0$-semigroup, and $F(t ): C_r \to \X$ for each $t$ and depends continuously and periodically on $t$  with period $1$, then the homogeneous equation
\begin{equation*}
\frac{du}{dt}=Au(t)+F(t)u_t,
\end{equation*}
will generates a $1$-periodic evolutionary process, denoted by $(U(t,s)_{t\ge s})$ in the phase space $C_r$.  In fact,
\begin{equation*}
U(t,s): C_r \ni \phi \mapsto u_t \in C_r ,
\end{equation*}
where $u$ is the solution of the equation
\begin{eqnarray*}
u(\tau ) &=& T(\tau -s) \phi (0) +\int^\tau _s T(\tau -\xi )F(\xi )u_\xi d\xi , \ \tau \ge s, \\
u_s &=& \phi .
\end{eqnarray*}
We introduce a function 
$\Gamma^n$ defined by 
$$\Gamma^n(\theta)=\left\{
\begin{array}{cc}
(n\theta+1)I,&\qquad -1/n\leq \theta\leq 0\\
&\\
0,&\theta<-1/n,
\end{array}
\right.
$$
where $n$ is any positive integer and $I$ is the identity operator on ${\mathbb X}$.  An explanation of the function $\Gamma^nf(s)$ as an element of $C_r$ is in order. By our definition,
$$
\Gamma^nf(s) : [-r,0] \ni \theta \mapsto \Gamma^n(\theta )f(s) =   \left\{
\begin{array}{cc}
(n\theta+1)f(s),&\qquad -1/n\leq \theta\leq 0\\
&\\
0,&\theta<-1/n,
\end{array}
\right .
$$

Since the evolutionary process $(U(t,s))_{t\ge s}$ is strongly continuous,
the ${C_r}$-valued function $U(t,s)\Gamma^nf(s)$ is 
continuous in $s\in [-r, t]$ whenever $f\in {\rm BUC}({\mathbb R^+},{\mathbb X}) .$ 

\medskip
The following theorem, whose proof could be found in \cite{murnaimin2}, is a variation of constant formula for solutions of (\ref{FDE}) in the phase space $C_r$:
\begin{theorem}\label{the vcf}
The segment $u_t(s,\phi;f)$ of solution $u(\cdot,s,\phi,f)$ of (\ref{FDE}) satisfies the following relation in $C_r$:
\begin{equation}\label{vcf}
u_t(s,\phi;f)=U(t,s)\phi
+\lim_{n\to \infty}\int_{s}^tU(t,\xi )\Gamma^nf(\xi )d\xi , \qquad t\ge s\ge 0 .
\end{equation}
Moreover, the above limit exists uniformly for bounded $|t-s |$.
\end{theorem}

Using the variation of constant formula (\ref{vcf}) we will prove the following estimate of spectrum that will be the key tool for our study in this paper.
\begin{lemma}\label{lem 4}
Let ${\cal M}=AP^+_{\Lambda}(\X) \oplus C_0(\X)$ be an admissible subspace of $BUC(\R^+,\X)$. Furthermore, let $\Lambda$ be of the form
\begin{equation*}
\Lambda = \{ \lambda \in \R | \ e^{i\lambda } \in \Xi \},
\end{equation*}
where $\Xi $ is a closed subset of the unit circle $ \Gamma$.

Then, under the above notations, for each function $u\in BUC(\R^+,\X)$ as a mild solution of Eq. (\ref{FDE}) on $\R^+$, the following estimate is valid:
\begin{eqnarray}\label{est 3}
\sigma_{\cal M} (u) &\subset & \sigma _\Gamma (P)\cup \sigma _{\cal M}(f).
\end{eqnarray}
\end{lemma}
\begin{proof}
By the formula (\ref{vcf}), for $t\ge 0$
\begin{eqnarray}\label{est-1}
u_{t+1} &=& U(t+1,t) u_{t} + \lim_{n\to \infty} \int^{t+1}_{t} U(t+1,s)\Gamma^n f(s)ds,
\end{eqnarray}
and the limit exists uniformly for all bounded $t$. Define the operator $G_n$ as below
\begin{equation*}
[G_n g ] (t):= \int^{t+1}_{t} U(t+1,s)\Gamma^n g(s)ds, \ g \in {\cal M} , t\ge 1.
\end{equation*}
To proceed, let us define an admissible subspace $\tilde{\cal M}$ in $BUC(\R^+, C_r)$ as
$$
\tilde{\cal M} := AP^+_{\Lambda}(C_r) \oplus C_0(C_r).
$$
Then, we consider the space $\tilde{\Y} := BUC(\R^+,C_r) /\tilde{\cal M}.$

\medskip
Note that $G_n$ commutes with $S$, so $\sigma_{\cal M} (G_nf) \subset \sigma _{\cal M}(f)$. Next, if we set
$$
[Gf](t) := \lim_{n\to \infty} \int^{t+1}_{t} U(t+1,s)\Gamma^n f(s)ds
,$$
then due to the uniformity of the limit $\sigma_{\cal M}(Gf) \subset \sigma_{\cal M}(f)$ as well.

At this point we can easily see that for each $\lambda \in \C$ such that $|\lambda |\not = 1$
\begin{eqnarray*}\label{101}
[(\lambda -S) u_{\cdot} ](t) &=& \lambda u_t - U(t+1,t)u_t + \lim_{n\to \infty} G_nf\\
&=& \lambda u_t - P(t)u_t + Gf
\end{eqnarray*}
Let us consider the operator of multiplication by $P(t)$, denoted by ${\cal P}$. The periodicity of the evolution process $(U(t,s))_{t\ge s}$ yields that $P(t)$ is 1-periodic, so it commutes with the translation $S$. Obviously
\begin{equation*}\label{resolvent}
(\lambda -\bar S ) \bar u_{\cdot} = (\lambda -\tilde {\cal P})\bar u_{\cdot} + Gf.
\end{equation*}
This means, 
\begin{equation*}
R(\lambda , \bar S)\bar u_\cdot  =R(\lambda , {\cal P})\bar u_\cdot  -R(\lambda ,{\cal P})R(\lambda , \bar S)\bar{Gf}),
\end{equation*}
whenever $\lambda$ is in a small neighborhood of $\xi_0 \not\in (\sigma_{\Gamma}(P)\cup \sigma _{\cal M}(f))$. This shows that $R(\lambda , \bar S)\bar u_\cdot $ has an analytic extension in a neighborhood of $\xi_0$, and so does $R(\lambda ,\bar S)\bar u$. This proves  (\ref{est 3}), completing the proof of the lemma..
\end{proof}

The following result is an analog of the Katznelson-Tsafriri Theorem.
\begin{theorem}\label{the 1}
Let $\sigma _\Gamma (P) \subset \{ 1\}$, and $u\in BUC(\R^+,\X)$ be a mild solution of Eq. (\ref{FDE}) on the half line $\R^+$, and let $f\in C_0(\X)$. Then,
\begin{equation*}
\lim_{t\to \infty} ( u(t+1)-u(t)) =0.
\end{equation*}
\end{theorem}
\begin{proof}
By Lemma \ref{lem 4}
\begin{eqnarray*}
\sigma (u) &\subset &\{1\} ,\\
\sigma (Su) &\subset& \{1 \} .
\end{eqnarray*}
Therefore,
\begin{equation*}
\sigma (Su-u) \subset \{ 1\} .
\end{equation*}
Consider the transform
\begin{equation*}
h(\lambda ):= R(\lambda , \bar S) (\bar S \bar u-\bar u).
\end{equation*}
It is analytic everywhere, with only possible exception at $\lambda =1$.

\medskip
On the other hand, using the identity $R(\lambda ,\bar S)\bar
S\bar u = \lambda R(\lambda , \bar S)\bar u - \bar x$ gives
\begin{eqnarray*}
 R(\lambda ,\bar S)(\bar S \bar u -\bar u ) &=& R(\lambda ,\bar S)\bar S \bar u-R(\lambda ,\bar S) \bar u \\
 &=& \lambda R(\lambda , \bar S)\bar u - \bar u -R(\lambda ,\bar S) \bar u    \\
 &=& (\lambda -1 )R(\lambda ,\bar S)  \bar u -\bar u .
\end{eqnarray*}
By Lemma \ref{lem 2} the function $h(\lambda )$ must be extendable analytically to the whole complex plane since $1$ is the only possible pole of first order of the function $R(\lambda ,\bar S)\bar u$. Therefore, by the definition of the spectrum of $ (Su-u)$, $\sigma (Su-u)=\emptyset$. By Lemma \ref{lem 1} $(Su-u ) \in C_0(\X)$. The theorem is proved.
\end{proof}

\begin{theorem}\label{the 2}
Let $u\in BUC(\R^+,\X)$ be a mild solution of Eq. (\ref{FDE}). Assume further that
\begin{enumerate}
\item $f\in AAP(\R^+,\X)$ with 
$$f=f_{AP}+f_0,$$
where $f_{AP}\in AP(\R^+,\X)$, and $f_0\in C_0(\X)$. 
\item
$\sigma_\Gamma (P)$ is countable and for every $x\in\X$ and $\xi_0\in \sigma _\Gamma (P)$
\begin{eqnarray}
\lim_{\lambda \downarrow \xi_0} (\lambda -\xi_0 )R(\lambda ,
P) x =0 .
\end{eqnarray}
\end{enumerate}
Then, the following assertions are true:
\begin{enumerate}
\item $u\in AAP(\R^+,\X)$, that is,
\begin{eqnarray}
u=u_{AP}+u_0,
\end{eqnarray}
where $u_{AP}\in AP(\R^+,\X)$ and
$u_0\in C_0(\X)$,
\item 
\begin{equation}\label{sp}
e^{\overline{\sigma_b(u_{AP})}}) \subset \sigma_\Gamma (P) \cup \overline{ e^{\overline{ \sigma_b (f_{AP}}}} )   .
\end{equation}
\end{enumerate}
\end{theorem}
\begin{proof}
By \cite[Theorem 2.10]{min} the sequence $\{ P^nu_0\} _{n=1}^{\infty}=\{u_n\}_{n=1}^{\infty}$ is asymptotically almost periodic. Applying the same lines of the proof of \cite[Proposition 2.11]{bathutrab} we can deduce that the mild solution $u(\cdot )$ itself is asymptotically almost periodic.

\medskip
Since in this case the space ${\cal M}=C_0(\X)$, 
\begin{equation*}
\sigma (u) =\sigma (u_{AP}), \ \ \sigma (f)=\sigma (f_{AP}).
\end{equation*}

\medskip
By Proposition \ref{pro} and by (\ref{bohr})
\begin{equation*}
\sigma (u_{AP}) =\overline{ e^{\overline{\sigma_b(u_{AP})}}}, \ \  \sigma (f_{AP}) =\overline{ e^{\overline{\sigma_b(f_{AP})}}} .
\end{equation*}
Now (\ref{sp}) follows from the estimate (\ref{est 3}) of Lemma \ref{lem 4}. This completes the proof of the theorem.
\end{proof}
The following is an immediate corollary of the above theorem:
\begin{corollary}
Assume that
\begin{enumerate}
\item 
$\sigma _\Gamma (P) \subset \{ 1\}$, $f$ is asymptotic 1-periodic function,
\item
\begin{eqnarray}
\lim_{\lambda \downarrow 1} (\lambda -1 )R(\lambda ,
P) x =0 .
\end{eqnarray}
\end{enumerate}
Then, every bounded uniformly continuous mild solution of (\ref{FDE}) is asymptotic 1-periodic.
\end{corollary}

\bibliographystyle{amsplain}

\end{document}